\documentclass[11pt,twoside,a4paper]{amsart}
\usepackage[top=2.54cm,bottom=3cm,left=2.0cm,right=2.54cm]{geometry} 
\usepackage[english]{babel}                      
\usepackage[utf8]{inputenc}
\usepackage[T1]{fontenc}     
\usepackage{amsmath,amssymb,amsthm}
\usepackage{setspace}
\usepackage{indentfirst}
\usepackage{paralist}
\usepackage{hyperref}
\usepackage{xcolor}

\renewenvironment{proof}{{\noindent \sc Proof:}}{\begin{flushright}$\blacksquare$\end{flushright}}

\newtheorem{teo}{Theorem}
\newtheorem{defi}{Definition}
\newtheorem{q}{Question}

\newtheorem*{claim}{Claim}
\theoremstyle{definition}
\newtheorem{ex}{Example}
\newtheorem*{notacao}{Notation}

\DeclareMathOperator{\spn}{span}
\DeclareMathOperator{\orb}{orb}

\newcommand{\norm}[1]{\left\lVert#1\right\rVert}
\newcommand{\abs}[1]{\left\lvert#1\right\rvert}
\newcommand{\eqdef}{\mathrel{\mathop:}=}
\newcommand{\restr}[2]{{\left.\kern-\nulldelimiterspace  #1\vphantom{\big|}\right|_{#2}}}

\title{Some Results on Subspace-Hypercyclic Operators}
\author{A. Augusto, L. Pellegrini}
\thanks{The research of the first author was supported by CNPQ, grant 142035/2018-1.}
\address{Departamento de Matemática, Instituto de Matemática e Estatística, Universidade de São Paulo}
\email{andreqa@ime.usp.br; leonardo@ime.usp.br}

\begin{document}
\begin{abstract}
A bounded linear operator $T$ on a Banach space $X$ is called subspace-hypercyclic if there is a subspace $M \subsetneq X$ and a vector $x \in X$ such that  $\orb{(x,T)} \cap M$ is dense in $M$. We show that every Banach space supports subspace-hypercyclic operators and provide a new criteira for subspace-hypercyclic operators, generalizing a previous result from Le \cite{le}.
\end{abstract}
\maketitle

\section*{Introduction}

A bounded linear operator $T$ on a separable Banach space $X$ is {\it hypercyclic} if there exists a vector $x \in X$ such that $\orb{(x,T)} \eqdef \{T^nx \, : \, n \geq 0\}$ is dense in $X$. Such vector $x$ is called a {\it hypercyclic vector} for $T$. Since Rolewicz \cite{rolewicz} constructed the first example of a hypercyclic operator on a Banach space, these operators have been massively studied. Amongst the results obtained since then, we highlight the Ansari-Bernal Theorem (every separable space admits a hypercyclic operator) and the prominent Hypercyclicity Criterion (a sufficient condition for hypercyclicity).  More information about these results and this topic can be found in \cite{LC}.

Recently, Madore and Martínez-Avendaño introduced in \cite{MMA} the concept of {\it subspace-hypercyclicity}: a bounded linear operator $T$ is subspace-hypercyclic if there is a subspace $M \subsetneq X$ and a vector $x \in X$ such that $\orb{(x,T)} \cap M$ is dense in $M$. In this case, we also say that $T$ is $M$-hypercyclic and that $x$ is a $M$-hypercyclic vector for $T$. Since the Madore and Martínez-Avendaño paper, this concept has been actively explored. One remarkable result was obtained in \cite{BKK}: there the authors prove that every hypercyclic operator is subspace-hypercyclic. 

It follows from the Ansari-Bernal theorem that every separable space admits a subspace-hypercyclic operator. However, the existence of subspace-hypercyclic operators on nonseparable Banach spaces hasn't been adressed until now. In this paper, we show that every Banach space (in particular, nonseparable spaces) admits a subspace-hypercyclic operator. Using this result, we will also show that, given an infinite-dimensional separable closed subspace $M$ in a Banach space $X$, there exists a bounded linear operator $T_M$ on $X$ such that $T_M$ is $M$-hypercyclic.

On Section 2, using an example from \cite{JMAP} as a blueprint, we obtain another criteria for subspace-hypercyclicity. This new criteria generalizes a previous result from Le \cite{le}.

\begin{notacao}
Throughout this paper $X$ will denote an infinite dimensional Banach space and $\mathcal{B}(X)$ the algebra of bounded linear operators on $X$.
\end{notacao}

\section{Existence of Subspace-Hypercyclic Operators}

Let us first recall some established definitions:

\begin{defi} Let $X, Y$ be Banach spaces and $T \in \mathcal{B}(X), S \in \mathcal{B}(Y)$. We say that $T$ is {\bf quasiconjugated} to $S$ (via $\phi$) if there exists a continuous map $\phi : X \to Y$ with dense range such that $\phi\circ T = S\circ\phi$.
\end{defi}

\begin{defi} An operator $T \in \mathcal{B}(X)$ is called {\bf weakly mixing} if the operator $T \oplus T$ is hypercylic on $X \oplus X$.
\end{defi}

It is an easy exercise to show that if $T$ is hypercyclic and quasiconjugated to $S$, then $S$ is hypercyclic. The same goes for the weakly-mixing property (and as well some other properties). Also, a well-known result from Bès and Peris \cite{equiv} states that the weakly-mixing property is equivalent to the Hypercyclicity Criterion.

With these definitions and aware of the results stated in the last paragraph, we may now prove the main theorem of this section:

\begin{teo} \label{teoex} Every infinite dimensional Banach space admits a subspace-hypercyclic operator.
\end{teo}
\begin{proof} The argument here is almost the same as the one used by Bonet and Peris in \cite{bonet}. As discussed in the introduction, we may only consider $X$ nonseparable.

Using a widely known result from Mazur about basic sequences in Banach spaces, let $M \subsetneq X$ be an infinite-dimensional separable closed subspace. Since $M$ is closed, we can look at $M$ as a Banach space itself. Hence, using a famous theorem from Ovsepian and Pełczyński \cite{ovsp}, we obtain sequences $\{x_n\}_{n = 1}^{\infty} \subseteq M$ and $\{x_n^*\}_{n = 1}^\infty \subseteq M^*$ such that:
\begin{enumerate}[(i)]
\item $x^*_i(x_j) = \delta_{ij}$
\item $\spn\{x_n \, : \, n \geq 1 \}$ is dense in $M$.
\item if $x \in M$ is such that $x^*_n(x) = 0, n \geq 1$, then $x = 0$.
\item $\norm{x_n} = 1, n \geq 1$ and $\sup_{n \geq 1}\norm{x^*_n} = C < \infty$.
\end{enumerate}

Using the Hahn-Banach theorem, we can extend each $x_n^*$ to $X$. To simplify what comes next, each extension will also be denoted as $x_n^* \in X^*$. Note that $\restr{x_n^*}{M}$ still satisfy conditions $(i), (iii)$ and $(iv)$ above.

Define $T \, : \, X \to X$ as $$Tx = x + \sum_{n = 1}^{\infty} 2^{-n}x^*_{n+1}(x)x_n$$  It is clear that $T$ is linear. Using the conditions above, we have that $T$ is bounded. Surely, by the Hahn-Banach theorem, we have that $\norm{x^*_{n+1}} = \norm{\restr{x_{n+1}^*}{M}} \leq C < \infty$. Then, if $x \in B_X$, we have
\begin{align*}
\norm{\sum_{n = 1}^{\infty} 2^{-n}x^*_{n+1}(x)x_n} & \leq \sum_{n = 1}^{\infty}\norm{2^{-n}x^*_{n+1}(x)x_n} =  \sum_{n = 1}^{\infty} 2^{-n}\abs{x^*_{n+1}(x)}\norm{x_n}  \\
& \leq \sum_{n = 1}^{\infty} 2^{-n}\norm{x^*_{n+1}}\norm{x} \leq \sum_{n = 1}^{\infty} 2^{-n}C < \infty
\end{align*} since, by the item $(iv)$, $\norm{x_n} = 1$. Hence $T$ is bounded.

Now we show that $T$ is $M$-invariant. Let $m \in M$. As $\spn\{x_n \, : \, n \geq 1 \}$ is dense in $M$, there is a sequence $(y_n)_{n \geq 1} \subseteq \spn\{x_n \, : \, n \geq 1 \} \subseteq M$ such that $y_n \to m$. Fix $n \geq 1$. Since $y_n = \sum_{j = 1}^{k} \lambda_jx_{n_j}$, we have:
\begin{align*}
\sum_{n = 1}^{\infty} 2^{-n}x^*_{n+1}(y_n)x_n & = \sum_{n = 1}^{\infty} 2^{-n}x^*_{n+1}\left(\sum_{j = 1}^{k} \lambda_jx_{n_j}\right)x_n = \sum_{n = 1}^{\infty}\sum_{j = 1}^{k} 2^{-n} \lambda_jx_{n+1}^*(x_{n_j})x_n \\
& = \sum_{j = 1}^{k} 2^{-(n_j - 1)} \lambda_jx_{n_j - 1}
\end{align*}

It is clear that  $\sum_{j = 1}^{k} 2^{-(n_j - 1)} \lambda_jx_{n_j - 1} \in \spn\{x_n \, : \, n \geq 1 \} \subseteq M$. Since $Ty_n =  y_n + \sum_{n = 1}^{\infty} 2^{-n}x^*_{n+1}(y_n)x_n$ and $y_n \in M$, this shows that $Ty_n \in M$. Now, as $T$ is continuous and $M$ is closed, we have that $T(m) \in M$, as desired.

Hence the operator $\restr{T}{M} \, : \, M \to M$ is well-defined. Let $S \, : \, \ell_1 \to \ell_1$ be $$S((\alpha_n)_{n \geq 1}) = \left(\alpha_1 + \frac{1}{2}\alpha_2, \alpha_2 + \frac{1}{2^2}\alpha_3, \alpha_3 + \frac{1}{2^3}\alpha_4, \ldots \right)$$

Note that $S$ is a mixing pertubation of the identity, hence hypercyclic.\footnote{See \cite[Corollary 8.3]{LC}.} Define $\phi \, : \, \ell_1 \to M$ as $$\phi((\alpha_n)_{n \geq 1}) = \sum_{n = 1}^{\infty} \alpha_nx_n$$

It is easy to see that $\phi$ is bounded and, looking back at the item $(iv)$ above, have dense range. Now, we have that:
\begin{align*}
\restr{T}{M}\circ\phi((\alpha_n)_{n \geq 1}) & = \restr{T}{M}\left(\sum_{n = 1}^{\infty} \alpha_nx_n\right) = \sum_{n = 1}^{\infty} \alpha_n\restr{T}{M}(x_n) =  \alpha_1x_1 + \sum_{n = 2}^{\infty} \alpha_n(x_n + 2^{-(n-1)}x_{n-1}) \\
& = \alpha_1x_1 + \sum_{n = 2}^{\infty}(\alpha_nx_n + \alpha_n2^{-(n-1)}x_n) =  \alpha_1x_1 + \sum_{n = 2}^{\infty} \alpha_nx_n +  \sum_{n = 2}^{\infty} \alpha_n2^{-(n-1)}x_{n-1} \\ 
& = \sum_{n = 1}^{\infty} \alpha_nx_n +  \sum_{n = 1}^{\infty} \alpha_{n+1}2^{-n}x_{n} =  \phi((\alpha_n)_{n \geq 1}) + \phi((2^{-n}\alpha_{n+1})_{n \geq 1})  \\
& = \phi((\alpha_n + 2^{-n}\alpha_{n+1})_{n \geq 1}) = \phi\circ S((\alpha_n)_{n \geq 1})
\end{align*}

Hence, $\restr{T}{M} \circ \phi = \phi \circ S$. Therefore, $S$ is quasiconjugate to $\restr{T}{M}$. Since $S$ is hypercyclic, then $\restr{T}{M}$ is hypercyclic. This shows that $T$ is $M$-hypercyclic.
\end{proof} 

Notice that if we were given the closed and separable subspace $M$ beforehand, we don't need to use Mazur's theorem at all - we can use the Ovsepian-Pełczyński theorem directly on that subspace $M$. With that in mind, we obtain an alternative version of our main theorem:

\begin{teo} Let $M \subsetneq X$ be an infinite-dimensional separable closed subspace. Then there is a bounded linear operator $T$ that is $M$-hypercyclic.
\end{teo}

In the case that $X$ is separable, it isn't clear if the operator constructed in the Theorem \ref{teoex} (that is obviously the same obtained in the theorem above) is hypercyclic. As the next theorem will show, if $X$ is separable, we have a better claim than the one provided by the last theorem:

\begin{teo} \label{newteosep} Let $X$ be separable and $M \neq X$ an infinite-dimensional closed subspace. Then there exists an invertible operator $T \in \mathcal{B}(X)$ such that $T$ is $M$-hypercyclic and it satisfies the Hypercyclicity Criterion.
\end{teo}
\begin{proof} By the Ansari-Bernal theorem, there exists an invertible operator $S \in \mathcal{B}(X)$ such that $S$ satisfies the Hypercyclicity Criterion. Hence, by a theorem from Subrahmonian Moothatu \cite[Theorem 5]{TKSM}, there exists an invertible operator $R \in \mathcal{B}(X)$ such that $S$ is $R(M)$-hypercyclic. Let $z$ be a $R(M)$-hypercyclic vector for $S$.

Consider now $x \eqdef R^{-1}(z)$. As $R^{-1}$ is continuous, we have $R^{-1}(\orb{(z,S)} \cap R(M)) = R^{-1}(\orb{(z,S)}) \cap M$ dense in $M$. Now, let $T \eqdef R^{-1}\circ S \circ R$. We have \begin{align*}
R^{-1}(\orb{(z,S)}) & = R^{-1}(\{z, Sz, S^2z, \ldots \}) = R^{-1}(\{Rx, SRx, S^2Rx, \ldots \}) = \\
& = \{x, R^{-1}SRx, R^{-1}S^2Rx, \ldots \}) = \orb{(x,T)}
\end{align*}

Hence $T$ is $M$-hipercyclic. By the definition of $T$, it is clear that $T$ is invertible. Also, since $T\circ R^{-1} = R^{-1}\circ S$ and $R^{-1}$ is continuous with dense range, it follows that $S$ is quasiconjugated to $T$. As $S$ is weakly mixing, so is $T$. Therefore $T$ satisfies the Hypercyclic Criterion.
\end{proof}

A well-known result from Grivaux \cite[Lemma 2.1]{grivaux} states that two countable dense sets of linearly independent vectors are linearly isomorphic. An immediate consequence of this result is that two hypercyclic operators have isomorphic orbits. Now we may ask if something similar is valid for subspace-hypercyclic operators: given that an operator $T$ is subspace-hypercyclic for both $M, N \subseteq X$, is there an invertible operator $A \, : \, X \to X$ such that $A(M) = N$? The answer to this question is no, as the next example shows:

\begin{ex} Let $X, Y$ be infinite dimensional separable Banach spaces that are non-isomorphic. Define $Z \eqdef X \oplus Y$, $M \eqdef X \oplus \{0\}$ and $N \eqdef \{0\} \oplus Y$. It is clear that $Z$ is a separable Banach space and $M, N \subseteq Z$ are closed subspaces. Using Theorem \ref{newteosep}, we find $T \in \mathcal{B}(Z)$ such that $T$ is $M$-hypercyclic. By a theorem from Subrahmonian Moothatu\footnote{The same one that we used in Theorem \ref{newteosep}.} \cite[Theorem 5]{TKSM}, there exists an invertible operator $R \in \mathcal{B}(Z)$ such that $T$ is $R(N)$-hypercyclic.

If the previous question had a positive answer, we would be able to find $A \in \mathcal{B}(Z)$ such that $A(R(N)) = M$. Hence, $(A \circ R)(N) = M$. Since $A \circ R$ is a composition of invertible operators, $A \circ R$ is invertible itself. Therefore, $N$ and $M$ are isomorphic - which entail that $X$ and $Y$ are isomorphic, a contradiction with the choice of both.

\vspace{1cm}
\end{ex}

On the other hand, we now may ask the following question:

\begin{q} If $T, S \in \mathcal{B}(X)$ are hypercyclic operators, does exist a subspace $M$ such that $T$ and $S$ are both $M$-hypercyclic?
\end{q}

The next theorem provides a partial answer to that question:

\begin{teo} Let $T, S$ be hypercyclic operators on a Banach space $X$. Then there is a subspace $M \subsetneq X$ and an invertible operator $A \, : \, X \to X$ such that $T$ is $M$-hypercyclic and $S$ is $A(M)$-hypercyclic.
\end{teo}
\begin{proof} Let $x$ and $y$ be hypercyclic vectors for $T$ and $S$, respectively. If $X_0 = \orb{(x,T)}$ and $Y_0 = \orb{(y,S)}$, since both sets are dense and linearly independent, then there exists an invertible operator $A \, : \, X \to X$ such that $A(X_0) = Y_0$ (this result is due to Grivaux \cite[Lemma 2.1]{grivaux}, as we mentioned before).

Using Theorem 2.1 from \cite{BKK}, we find a subspace $M\subsetneq X$ such that $X_0 \cap M$ is dense in $M$. We now need to show that $Y_0 \cap A(M)$ is dense in $A(M)$. Since $A$ is invertible, we have that $A(M) = A(\overline{X_0 \cap M}) = \overline{A(X_0 \cap M)}$. Moreover, since $A(X_0 \cap M) \subseteq A(M)$, we have $A(X_0 \cap M) = A(X_0 \cap M) \cap A(M)$. Putting everything together, we have that $A(M) = \overline{A(X_0 \cap M) \cap A(M)}$.

Finally, since $X_0 \cap M \subseteq X_0$, then $A(X_0 \cap M) \subseteq A(X_0) = Y_0$. Hence, we have $A(X_0 \cap M) \cap A(M) \subseteq Y_0 \cap A(M) \subseteq A(M)$. Therefore, $$A(M) = \overline{A(X_0 \cap M) \cap A(M)} \subseteq \overline{Y_0 \cap A(M)} \subseteq A(M)$$ which shows that  $Y_0 \cap A(M)$ is dense in $A(M)$, as desired.
\end{proof}

\section{A New Criteria for Subspace-Hypercyclicity}

When Madore and Martínez-Avendaño introduced the concept of subspace-hypecyclicity in \cite{MMA}, they immediately proved a Subspace-Hypercyclicity Criterion, clearly based on its hypercyclic counterpart. Later, Can Le made in \cite{le} another criteria for subspace-hypercyclicity. The difference between both criteria is simple: the conditions imposed on Le's Criterion are more strict (for example, injective operators can't satisfy Le's Criterion) but it's claim is way better, as we shall see later.

In \cite{JMAP}, the authors provided an example of a subspace-hypercyclic operator $T$ such that $\orb{(x,T)} \cap M$ is somewhere dense in $M$ but not everywhere dense in $M$. Their example helped us devise a new criteria for subspace-hypercyclicity. Before showing our new criteria, we first need the following definition:

\begin{defi} If $T$ is a bounded linear operator, then the {\bf generalized kernel of $T$} is defined as $$\ker^*(T) \eqdef \bigcup_{m = 1}^{\infty} \ker(T^m)$$
\end{defi}

\begin{teo} \label{newteokert} Let $T \in \mathcal{B}(X)$. Assume that there exists an infinite-dimensional separable closed subspace $M$ such that $\ker^*(T) \cap M$ is dense in $M$, a map $A \, : \, \ker^*(T) \to \ker^*(T)$ and an increasing sequence $(m_k)_{k \geq 1} \subseteq \mathbb{N}$ such that:
\begin{enumerate}[(i)]
\item $A^{m_k}x \to 0$ for every $x \in \ker^*(T) \cap M$.
\item $A^{m_k}x \in \ker^*(T) \cap M$ for every $x \in \ker^*(T) \cap M$.
\item $m_j - m_i \in (m_k)_{k \geq 1}$ for all $i < j$.
\item $(T\circ A)x = x$ for all $x \in \ker^*(T) \cap M$.
\end{enumerate}

Then $T$ is $M$-hypercyclic.
\end{teo}
\begin{proof} Let $(x_k)_{k \geq 1} \subseteq \ker^*(T) \cap M$ be a dense sequence in $M$. 

\begin{claim} There exists an increasing subsequence $(m_{j_k})_{k \geq 1} \subseteq (m_k)_{k \geq 1}$ such that $$\norm{A^{m_{j_k}}x_k} < 2^{-k}, \, \, \, \, \norm{A^{m_{j_k + i}}x_{k+1}} < 2^{-(k+1)},  \, \, \, \, T^{m_{j_k}}x_k = 0 \, \, \, \, \text { and }  \, \, \, \, \frac{m_{j_{k+1}}}{m_{j_k}} \geq 2$$ for all $i \geq 1$.
\end{claim}

Indeed, let us start with $k = 1$. Given $\varepsilon = 2^{-1}$, using $(i)$ there are $m_{k_1}, m_{k^*_1}$ such that $\norm{A^{m_j}x_1} < 2^{-1}$, for all $j \geq k_1$, and $\norm{A^{m_j}x_2} < 2^{-2}$, for all $j \geq k_1^*$. Since $x_1 \in \ker^*(T)$, then there exists $p_1 \geq 1$ such that $x_1 \in \ker(T^{p_1})$. Choosing $m_{j_1} \geq \max\{m_{k_1}, m_{k_1^*}, p_1\}$, it is easy to see that $\norm{A^{m_{j_1}}x_1} < 2^{-1}, \norm{A^{m_{j_1 + i}}x_2} < 2^{-2}$  and $T^{m_{j_1}}x_1 = 0$.

Suppose now that we have constructed $(m_{j_k})_{k = 1}^{t-1}$. Just like we did it before, we can find $m_{k_t}, m_{k_t^*}$ such that $\norm{A^{m_j}x_t} < 2^{-t}$, for all $j \geq k_t$ and $\norm{A^{m_j}x_{t+1}} < 2^{-(t+1)}$, for all $j \geq k_t^*$. Also, there is $p_t \geq 1$ such that $x_t \in \ker(T^{p_t})$. Taking $m_{j_t} \geq \max\{m_{k_t}, m_{k_t^*}, p_t, 2m_{j_{t-1}}\}$, we have that $m_{j_t}$ satisfies all four desired conditions.

Note now that $\sum_{k = 1}^{\infty} A^{m_{j_k}}x_k$ is absolutely convergent since $\norm{A^{m_{j_k}}x_k} < 2^{-k}$ for all $k \geq 1$. Denoting  $x \eqdef \sum_{k = 1}^{\infty}A^{m_{j_k}}x_k$, we will now show that $T^{m_{j_k}}x \in M$ and $\{T^{m_{j_k}}x \, : \, k \geq 1\}$ is dense in $M$. This clearly shows that $T$ is $M$-hypercyclic, as desired. \\ 

\begin{itemize}
\item $T^{m_{j_k}}x \in M$.
\end{itemize}

Fix $k \geq 1$. Using condition $(iv)$, we have that
\begin{equation} \label{teoleeq}
T^{m_{j_k}}x = \sum_{i = 1}^{k-1}T^{m_{j_k} - m_{j_i}}x_i + x_k + \sum_{i = k+1}^{\infty}A^{m_{j_i} - m_{j_k}}x_i
\end{equation}

Fix an $1 \leq i \leq k-1$. Since ${m_{j_{k+1}}}/{m_{j_k}} \geq 2$, then ${m_{j_k}} - {m_{j_i}} \geq {m_{j_i}}$. Hence, $T^{m_{j_k} - m_{j_i}}x_i = 0$ (because $T^{m_{j_i}}x_i = 0$ for all $i \geq 1$) and therefore $$\sum_{i = 1}^{k-1}T^{m_{j_k} - m_{j_i}}x_i = 0$$

\noindent Now, if $i > k$, by the condition $(iii)$ we have that $m_{j_i} - m_{j_k} \in (m_k)_{k \geq 1}$. Hence, using condition $(ii)$, it follows that $A^{m_{j_i} - m{j_k}}x_i \in \ker^*(T) \cap M$. Then $$\sum_{i = k+1}^{\infty}A^{m_{j_i} - m_{j_k}}x_i \in \overline{\ker^*(T) \cap M} = M$$ 

\noindent Since $x_k \in M, \sum_{i = k+1}^{\infty}A^{m_{j_i} - m_{j_k}}x_i \in M$ and $\sum_{i = 1}^{k-1}T^{m_{j_k} - m_{j_i}}x_i = 0$, it follows by (\ref{teoleeq}) that $T^{m_{j_k}}x \in M$. \\ 
 
\begin{itemize}
\item $\{T^{m_{j_k}}x \, : \, k \geq 1\}$ is dense in $M$.
\end{itemize}

Since $(x_k)_{k \geq 1}$ is dense in $M$, it is enough to show that $\norm{T^{m_{j_k}}x - x_k} < 2^{-k}$. Using (\ref{teoleeq}), we have:
\begin{align*}
\norm{T^{m_{j_k}}x - x_k} & = \norm{\sum_{i = 1}^{k-1}T^{m_{j_k} - m_{j_i}}x_i + \sum_{i = k+1}^{\infty}A^{m_{j_i} - m_{j_k}}x_i} \\
& \leq \sum_{i = 1}^{k-1}\norm{T^{m_{j_k} - m_{j_i}}x_i} + \sum_{i = k+1}^{\infty}\norm{A^{m_{j_i} - m_{j_k}}x_i}
\end{align*}

\noindent As we saw earlier, we have that $\sum_{i = 1}^{k-1}T^{m_{j_k} - m_{j_i}}x_i = 0$. \medskip Fix any $i > k$. Then $m_{j_i} > m_{j_{i-1}} \geq m_{j_k}$. Since ${m_{j_{k+1}}}/{m_{j_k}} \geq 2$ for all $k \geq 1$, we have that $$m_{j_i} - m_{j_k} \geq m_{j_i} - m_{j_{i-1}} \geq m_{j_{i-1}}$$

By condition $(iii)$, we have $m_{j_i} -  m_{j_k} \in (m_k)_{k \geq 1}$. Hence, denoting $m_{j_i} -  m_{j_k} = m_l$, from the last inequality we obtain $m_l > m_{j_{i-1}}$. If $a_i \geq 1$ is such that  $l = j_{i-1} + a_i$, we obtain $$\norm{A^{m_{j_i} - m_{j_k}}x_i} = \norm{A^{m_l}x_{i}} = \norm{A^{m_{j_{i-1}+ a_i}}x_{i}}  < 2^{-i}$$ by construction. Finally:
\begin{align*}
\norm{T^{m_{j_k}}x - x_k} & \leq \sum_{i = 1}^{k-1}\norm{T^{m_{j_k} - m_{j_i}}x_i} + \sum_{i = k+1}^{\infty}\norm{A^{m_{j_i} - m_{j_k}}x_i} \\
& \leq \sum_{i = k+1}^{\infty} 2^{-i} = 2^{-k}
\end{align*} 
\end{proof}

It's not hard to see that, if $T$ and $M$ satisfy our criteria and $X$ is separable, then $T$ and $M$ also satisfy the Subspace-Hypercyclicicty Criterion. Hence, it's fair to ask about the usefulness of the above criteria. As we said in the introduction, this criteria generalizes a result from Le \cite{le}. In order to facilitate what we are going to discuss, let us first recall the aforementioned result:

\begin{teo}[Le's Criterion, \cite{le}] \label{teolekert} Let $T$ be a bounded linear operator on a separable Banach space $X$ such that $\ker^*(T)$ is dense in $X$ and there exists a map $A \, : \, \ker^*(T) \to \ker^*(T)$ satisfying
\begin{enumerate}[(1)]
\item $A^mx \to 0$ for every $x \in \ker^*(T)$,
\item $TA = I$ on $\ker^*(T)$.
\end{enumerate}

Then $T$ is $M$-hypercyclic for all finite co-dimensional subspaces $M$.
\end{teo}

In addition, it can be easily seen that Theorem \ref{newteokert} can be used in any Banach space, whereas Le's Criterion can only be applied on separable spaces. However, this isn't the reason why Theorem \ref{newteokert} is a generalization of Le's result, as we can make a version of Le's result that works for nonseparable Banach spaces. 

Indeed, let $X$ be a nonseparable Banach space, $S \in \mathcal{B}(X)$ and $M \subseteq X$ an infinite-dimensional separable closed subspace. Suppose that $\ker^*(S) \subseteq M$ is dense in $M$. Hence, we have that $S$ is $M$-invariant.\footnote{Indeed, if $m \in M$ and $\ker^*(S) \subseteq M$ is dense in $M$, then there is $(x_n)_{n \geq 1} \in \ker^*(S)$ such that $x_n \to m$. Since $S(x_n) \in \ker^*(S)$ for every $n \geq 1$ and $S(x_n) \to S(m)$ then $S(m) \in \overline{\ker^*(S)} = M$.} 

Suppose now that there exists $A_1 \, : \, \ker^*(S) \to \ker^*(S)$ such that $A_1$ satisfies conditions $(1)$ and $(2)$ of Le's Criterion. Since $S$ is $M$-invariant, the operator $\restr{S}{M} \, : \, M \to M$ is well-defined. Hence, taking $X \eqdef M$, $T \eqdef \restr{S}{M}$ e $A \eqdef A_1$ on Le's Criterion, then we have that $T = \restr{S}{M}$ is $N$-hypercyclic for every subspace $N$ that have finite codimension in $M$ - which means that $S$ is $N$-hypercyclic for these subspaces as well. Not only that, Le noticed in his paper \cite{le} that an operator that satisfies $(i)$ and $(ii)$ on Theorem \ref{teolekert} is hypercylic. Since $\restr{S}{M}$ satisfies both conditions, then $\restr{S}{M}$ is hypercyclic. Hence $S$ is also $M$-hypercyclic.

Therefore, we can obtain this natural generalization of Le's theorem for nonseparable Banach spaces. With that in mind - and the fact that every operator that satisfies this ``general version'' of Le's theorem is $M$-invariant (as showed above), we have this easy example to show that our Theorem \ref{newteokert} is, as said before, a proper generalization for nonseparable spaces:

\begin{ex} \label{exemplo} Let $X = \ell_\infty$ and $T = 2B$, where $B$ is the widely known backward shift operator on $\ell_\infty$. Let $M = \{(a_n) \, : \, a_n \to 0$ and $a_{2n} = 0\}$. It is clear that $M$ is separable and a closed subspace of $\ell_\infty$. Note that $T$ isn't $M$-invariant, hence it doesn't satisfy the general version of Le's Criterion. 

We have that $\ker^*(T) \cap M$ is clearly dense in $M$. Therefore, if $F$ is the forward shift operator on $\ell_\infty$, taking $A = 2^{-1}F$ and $(m_k)_{k \geq 1} = (2k)_{k \geq 1}$, we have that $A$ and $(m_k)_{k \geq 1}$ both satisfy the four conditions of Theorem \ref{newteokert}. Hence, $T$ is $M$-hypercyclic.

\vspace{1cm}
\end{ex}

\section{Final Remark}

This article is part of the first author's PhD thesis, written under the supervision of the second author.

\end{document}